\documentclass{article}

\usepackage{graphicx}
\usepackage{amsmath}
\usepackage{amsthm}
\usepackage{amssymb}
\usepackage{amsfonts}
\usepackage{tikz-cd}
\usepackage[all]{xy}
\usepackage{enumerate}
\usepackage{multicol}
\usepackage[colorlinks=true,citecolor=red,linkcolor=blue]{hyperref}

\newtheorem{theorem}{Theorem}[section]
\newtheorem{lemma}[theorem]{Lemma}
\newtheorem{proposition}[theorem]{Proposition}
\newtheorem{corollary}[theorem]{Corollary}
\newtheorem{question}[theorem]{Question}
\newtheorem{problem}[theorem]{Problem}

\newtheorem{definition}[theorem]{Definition}

\newtheorem{notation}[theorem]{Notation}

\newtheorem{remark}[theorem]{Remark}

\newcommand{\Z}{\mathbb{Z}}
\newcommand{\Q}{\mathbb{Q}}
\newcommand{\R}{\mathbb{R}}
\newcommand{\C}{\mathbb{C}}

\newcommand{\Fcal}{\mathcal{F}}

\newcommand{\Kcal}{\mathcal{K}}

\newcommand{\Ocal}{\mathcal{O}}

\newcommand{\JR}{{\rm JR }}

\newcommand{\Gal}{{\rm Gal }}

\DeclareMathOperator{\Aut}{Aut}

\begin{document}

\title{An approach to Julia Robinson numbers through the lattice of subfields}
\author{Xavier Vidaux and Carlos R. Videla}

\maketitle

\begin{abstract}
By fully describing the lattice of subfields of some towers of number fields built by iterating square roots, we obtain infinitely many fields, each of them either contradicts Julia Robinson's problem (obtaining a $\JR$-number $4$ which is not a minimum) or gives a Julia Robinson number strictly between four and infinity. This improves a previous result by M. Castillo and the same authors.\footnote{The two authors have been partially supported by the first author's ANID Fondecyt research projects 1170315 and 1210329, Chile. The first author was fully granted by the VRID, University of Concepci\'on, Chile, for a visit in April 2022 to Mount Royal University. He also thanks Mount Royal University for its hospitality during the stay. }
\end{abstract}




\section{Introduction}

In 1959, Julia Robinson \cite[Problem 5]{Rob59} --- see also \cite[Corollary and the remark below, p. 95]{Rob62} --- raised a problem about the distribution of conjugate sets of integers in the ring of integers of a totally real algebraic extension of the rationals. 

\begin{problem}\label{jr}
Show that in any totally real algebraic field, there is an interval $0<x<t$ ($t$ may be $+\infty$) containing infinitely many sets of conjugates of algebraic integers of the field, and such that there is only a finite number of sets of conjugate integers for any smaller $t$. 
\end{problem}

She proved that if one can solve this problem for a given field $K$, then the first order theory of the ring of integers of this field is undecidable. Since Problem \ref{jr} is still open, we reformulate it as a question: 

\begin{question}\label{jr2}
Is it true that the $t$ mentioned in Problem \ref{jr} actually exists for \emph{any} totally real field?
\end{question}

In a previous work \cite{VV15}, we approached this question by studying certain extensions of $\Q$ built by iterating square roots. Fields obtained by iterating quadratic polynomials have been studied for a long time and for various purposes --- see for instance \cite{Od85,St92,Ya20,Li22,Sm23}. In our context, we discovered that for certain subrings of the ring of integers of the fields that we consider, the $t$ does not exist. This does not answer Question \ref{jr2}, because J. Robinson had in mind the full ring of integers of the field, and we did not know then, nor now, whether any of the subrings that we considered is the ring of integers of its fraction field. Following  \cite{VV15}, it is therefore natural to reformulate the question in the following way. Given a ring $R$ of totally real algebraic integers, and $t\in\R\cup\{+\infty\}$, consider the set 
$$
\Ocal_{t}=\{x\in R\colon \textrm{all the conjugates of $x$ lie in the interval }(0,t)\},
$$ 
and define the $\JR$-number of $R$ as the infimum of the set of $t$ such that $\Ocal_t$ is infinite. Note that Julia Robinson's question has a positive answer for $R$ if and only if the $\JR$-number of $R$ is a minimum, if and only if the set of $t$ such that $\Ocal_t$ is infinite is a closed interval or $\{+\infty\}$. So J. Robinson's question is whether or not the $\JR$-number of the ring of integers of a totally real field is always a minimum. It is convenient to define the $\JR$-number of a field of totally real numbers as the $\JR$-number of its ring of integers. 

The $\JR$-number of some of the subrings from \cite{VV15} mentioned above are not minimum. Marianela Castillo \cite{Ca21}, in her PhD thesis, proved that \emph{any} of the subrings considered in \cite{VV15} is the ring of integers of its quotient field if and only if a certain sequence of positive integers associated to the field is always square free. However, it is still an open problem to find the precise $\JR$-number of the integral closure of these subrings in their quotient field. 

In \cite{CVV20}, we proved that for many of these fields, the $\JR$-number either is equal to $4$ and is not a minimum, or is strictly between $4$ and $+\infty$. In fact, there were no known $\JR$-numbers of fields strictly between $4$ and $+\infty$. The latter was solved by P. Gillibert and G. Ranieri \cite{GR19} --- for all the fields that they construct, the $\JR$-number is a minimum. More generally, the problem of determining which real numbers can be realized as the $\JR$-number of a ring of totally real integers was asked in \cite{VV15}. The only known $\JR$-numbers for rings of integers are quadratic irrationalities. For example, at the present it is not known whether the $\JR$-number can be a cubic irrationality. This distribution problem has recently attracted interest in other contexts --- see for instance \cite{VV16, Wi16, PTW22, OS22,Ok22,OS23}.  

In this paper, we completely determine the structure of subfields of the fields considered in \cite{VV15}, allowing us to improve the main result of \cite{CVV20}, by getting rid of most of its hypothesis. We believe that knowing the structure of the lattice of subfields could be useful to determine the $\JR$-number of these fields, and might lead to a counter-example for Julia Robinson's problem (as far as we know, the fields that we consider are the only known ones for which there is some hope to obtain such counter-examples).

The infinite extensions that we consider are built in the following way. For integers $\nu\ge2$ and $x_0\ge0$, consider $x^{\nu,x_0}_{n+1}=x_{n+1}=\pm\sqrt{\nu+x_n}$ for each $n\ge0$ (any choice of sign can be taken at every step), $K^{\nu,x_0}_n=\Q(x_n)$ and $K^{\nu,x_0}=\cup_{n\ge0} K^{\nu,x_0}_n$. Let $\Omega$ be the set of pairs $(\nu,x_0)$ such that for each $n\ge0$, $[K^{\nu,x_0}_{n+1}:K^{\nu,x_0}_n]=2$, and $K^{\nu,x_0}$ is totally real (we know from \cite{VV15} that $\Omega$ is infinite --- see \cite[Ch. 2]{Ca18} for more pairs in $\Omega$). 

Write $\Ocal^{\nu,x_0}$ for the ring of integers of $K^{\nu,x_0}$, and $\Z^{\nu,x_0}$ for the union over $n\ge1$ of the rings $\Z[x_n]$.

 We are now in condition to state our two main theorems. The following theorem shows that the conclusion of \cite[Thm 1.1]{CVV20} is valid for essentially all pairs $(\nu,x_0)$. 

\begin{theorem}\label{main2}
Given $(\nu,x_0)\in\Omega$, distinct from $(2,0)$ and $(2,1)$, the $\JR$-number of $K^{\nu,x_0}$ either is $4$ and it is not a minimum, or it is strictly between $4$ and $+\infty$. 
\end{theorem}

Observe that the $\JR$-number of $K^{\nu,x_0}$ is at most the $\JR$-number of $\Z^{\nu,x_0}$, which is finite --- see \cite{VV15,Ca18}. Theorem \ref{main2} is a consequence of the following theorem and its two corollaries, as will be explained below. 

\begin{theorem}[Determination of the lattice of subfields]\label{main}
Given $(\nu,x_0)\in\Omega$, let $u_0=\nu^2-\nu$ and $u_{n+1}=u_{n}^2-\nu$ for every $n\ge0$. We have: 
\begin{enumerate}
\item If $u_0-x_0$ is not a square, then the only proper subfields of $K^{\nu,x_0}$ are the $K^{\nu,x_0}_n$. 
\item If $u_0-x_0=a^2$ is a square and $\nu\ge3$, then the only proper subfields of $K^{\nu,x_0}$ are the $K^{\nu,x_0}_n$, and the two following quadratic extensions of $\Q$ that lie in $K^{\nu,x_0}_2$: $\Q(\sqrt{2(\nu-a)})$ and $\Q(\sqrt{2(\nu+a)})$. This happens for infinitely many pairs $(\nu,x_0)$ in $\Omega$. See Figure \ref{fig:SubCampNot21Intro}. 
\item The only proper subfield of $K^{2,1}$ of infinite degree over $\Q$ is $K^{2,0}$. We have $K^{2,1}=K^{2,0}(\sqrt{3})$, and for each $n\ge1$, $K^{2,1}_{n+1}=K^{2,0}_n(\sqrt{3})$, with $K^{2,0}_n\ne K^{2,1}_{n+1}$. Moreover, the lattice of subfields of $K^{2,1}$ has the structure described in Figure \ref{fig:SubCamp21Intro}, where $M_n=\Q(\sqrt3 x^{2,0}_n)$ and each (non dotted) line corresponds to a degree $2$ extension. 
\end{enumerate}
\end{theorem}

\begin{corollary}\label{CorKnux020}
If $K^{2,0}\subseteq K^{\nu,x_0}$, then $(\nu,x_0)=(2,0)$ or $(\nu,x_0)=(2,1)$. 
\end{corollary}

Given an integer $m$, let $\zeta_m$ denote a primitive $m$-th root of unity. 

\begin{corollary}\label{CorFermat}
If $\zeta_m+\zeta_{m}^{-1}$ is in $K^{\nu,x_0}$ for some $m$, then $m$ has the form either $2^rp_1p_2$ for $r\le2$, or $2^rp_1$ for $r\ge3$, or $2^r$ for $r\ge2$, where $p_1$ and $p_2$ are distinct Fermat primes. 
\end{corollary}

Theorem \ref{main2} follows. Indeed, Following \cite[Section 2]{CVV20}, we know that the $\JR$-number is $4$ and is a minimum if and only in $\Ocal^{\nu,x_0}$ there are infinitely many numbers of the form $\zeta_m+\zeta_m^{-1}$. By Theorem \ref{main}, the field $K^{\nu,x_0}$ has at most three quadratic subextensions, so by Corollary \ref{CorFermat}, there are at most finitely many possible $m$ that are not of the form $2^r$, $r\ge2$, such that $\zeta_m+\zeta_{m}^{-1}$ is in $K^{\nu,x_0}$, since square-roots of prime numbers are linearly independent over $\Q$. By Corollary \ref{CorKnux020}, there are only finitely many possible $m$ of the form $2^r$. So there are only finitely many possible numbers of the form $\zeta_m+\zeta_m^{-1}$, hence the conclusion.

Note that item 3 of Theorem \ref{main} shows in particular that the structure of subfields of $K^{2,1}_{n}$ is the same as that of the cyclotomic field $\Q(\zeta_{2^{n}})$ (see \cite[Prop. C12 and Fig. C6, page 120--121]{Mi14}).

In order to prove Theorem \ref{main}, we prove some general results about $2$-towers of number fields that may be of independent interest: the existence of subfields different from $K_n$ is reflected in the Galois groups of certain quartic extensions within the tower ---  see Theorem \ref{generictower} and Corollary \ref{lemgenericsubtower}. 

In Section \ref{secThinnessGeneral}, we define the concept of thinness for $p$-towers, and give a characterization of thinness for $2$-towers in general. In Section \ref{NotFacts}, we introduce notation, state some known facts about the towers that are involved in Theorem \ref{main}, and some general facts about quartic extensions. In Section \ref{secQuarticGen}, we study the quartic extensions that lie within our towers $(K_n)$. 

We prove item 2 in Section \ref{secOmega1} using Siegel's finiteness Theorem (see Lemma \ref{lemEllCurves}). We prove item 3 in Section \ref{subsecnu2}. Finally, in Section \ref{secSqrt2}, we characterize the pairs $(\nu,x_0)$ such that $\sqrt{2}\in K$, the pairs such that $\sqrt{2+\sqrt{2}}\in K$ and we prove Corollaries \ref{CorKnux020} and \ref{CorFermat}. 

Inspired by Corollary \ref{CorKnux020}, we ask:
\begin{question}\label{quest}
For which pairs $(\nu,x_0)$ and $(\nu',x_0')$ do we have $K^{\nu,x_0}\subseteq K^{\nu',x_0'}$?
\end{question}

The smallest known $\JR$-number of a field is $\lceil 2\sqrt{2}\rceil +2\sqrt{2}\approx 5.828$, which was obtained in \cite{GR19}. The $\JR$-number of $\Z^{4,3}$ is $\lfloor\alpha\rfloor+\alpha+1\approx 5.562$, with $\alpha=\frac{1+\sqrt{17}}{2}$ --- see \cite{VV15}. So the $\JR$-number of $K^{4,3}$ is less than $5.562$. 
\begin{problem}\label{prob}
Find the $\JR$-number of $K^{4,3}$. 
\end{problem}

As a last comment, note that most of the ingredients in this paper extend straightforward to the non-totally real case, but in general the structure of subfields outwits us.

\begin{figure}[htp]
\centering
\begin{tikzpicture}[scale=1.0]
\draw (0,0) node{$\Q$}; 
\draw (0,0.3) -- (0,0.7); \draw (0,1) node{$K^{\nu,x_0}_1$}; 
\draw (0,1.3) -- (0,1.7); \draw (0,2) node{$K^{\nu,x_0}_2$}; 
\draw (0,2.3) -- (0,2.7); \draw (0,3) node{$K^{\nu,x_0}_3$}; 
\draw [dotted] (0,3.3) -- (0,4.0); \draw (0,4.3) node{$K^{\nu,x_0}$}; 
\draw (-3,1) node{$\Q\left(\sqrt{2(\nu-a)}\right)$};
\draw (-0.4,0.1) -- (-1.6,0.8); \draw (-1.6,1.1) -- (-0.4,1.8); 
\draw (3,1) node{$\Q\left(\sqrt{2(\nu+a)}\right)$};
\draw (0.4,0.1) -- (1.6,0.8); \draw (0.4,1.8) -- (1.6,1.1); 
\end{tikzpicture}
\caption{The lattice of subfields of $K^{\nu,x_0}$ when $u_0-x_0$ is a square and $\nu\ge3$.}\label{fig:SubCampNot21Intro}
\end{figure}
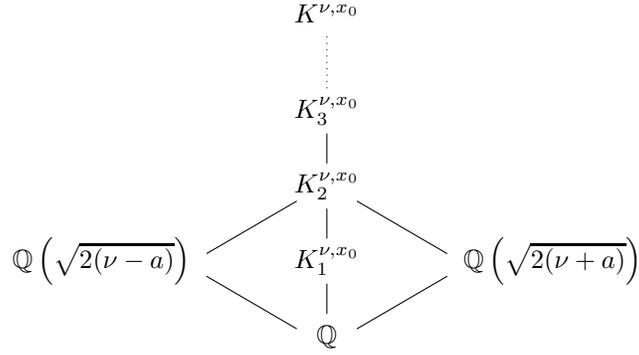

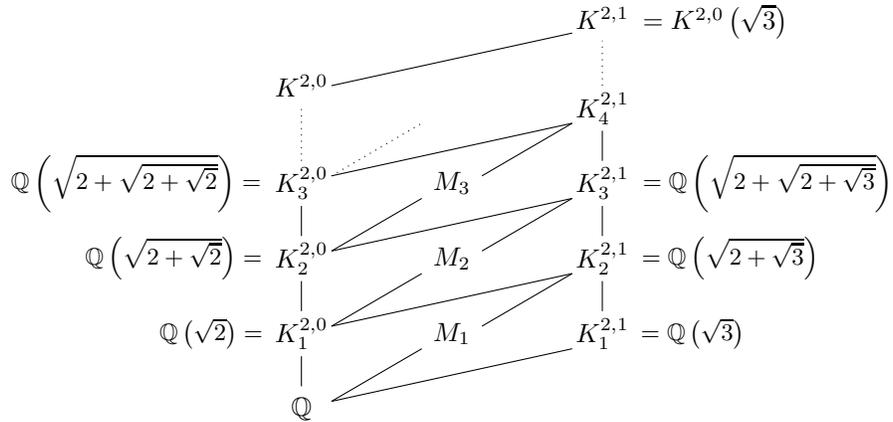
\begin{figure}[htp]
\centering
\begin{tikzpicture}[scale=1.0]
\draw (-2,0) node{$\Q$}; 
\draw  (-1.6,0.1) -- (-0.4,0.8); 
\draw (0,1) node{$M_1$}; \draw (0,2) node{$M_2$}; \draw (0,3) node{$M_3$}; 
\draw (-1.6,0.1) -- (1.6,0.8); \draw (2,1) node{$K^{2,1}_1$}; 
\draw (3.2,1) node{\text{\small $=\Q\left(\sqrt3\right)$}};
\draw (2,1.3) -- (2,1.7); \draw (2,2) node{$K^{2,1}_2$}; 
\draw (3.7,2) node{\text{\small $=\Q\left(\sqrt{2+\sqrt{3}}\right)$}};
\draw (2,2.3) -- (2,2.7); \draw (2,3) node{$K^{2,1}_3$}; 
\draw (4.2,3) node{\text{\small $=\Q\left(\sqrt{2+\sqrt{2+\sqrt{3}}}\right)$}};
\draw (2,3.3) -- (2,3.7); \draw (2,4) node{$K^{2,1}_4$}; 
\draw [dotted] (2,4.3) -- (2,4.9); \draw (2,5.2) node{$K^{2,1}$}; 
\draw (3.5,5.15) node{$=K^{2,0}\left(\sqrt3\right)$};
\draw (-2,0.3) -- (-2,0.7); \draw (-2,1) node{$K^{2,0}_1$}; 
\draw (-3.2,1) node{\text{\small $\Q\left(\sqrt2\right)=$}};
\draw (-2,1.3) -- (-2,1.7); \draw (-2,2) node{$K^{2,0}_2$}; 
\draw (-3.7,2) node{\text{\small $\Q\left(\sqrt{2+\sqrt{2}}\right)=$}};
\draw (-2,2.3) -- (-2,2.7); \draw (-2,3) node{$K^{2,0}_3$}; 
\draw (-4.2,3) node{\text{\small $\Q\left(\sqrt{2+\sqrt{2+\sqrt{2}}}\right)=$}};
\draw [dotted] (-2,3.3) -- (-2,4.0); \draw (-2,4.3) node{$K^{2,0}$}; 
\draw (0.4,1.1) -- (1.6,1.8); 
\draw (0.4,2.1) -- (1.6,2.8); 
\draw (0.4,3.1) -- (1.6,3.8); 
\draw (-1.6,1.1) -- (-0.4,1.8); 
\draw (-1.6,2.1) -- (-0.4,2.8); 
\draw [dotted] (-1.6,3.1) -- (-0.4,3.8); 
\draw (-1.6,1.1) -- (1.6,1.8); 
\draw (-1.6,2.1) -- (1.6,2.8); 
\draw (-1.6,3.1) -- (1.6,3.8); 
\draw (-1.6,4.3) -- (1.6,5); 
\end{tikzpicture}
\caption{The lattice of subfields of $K^{2,1}$}\label{fig:SubCamp21Intro}
\end{figure}

\section{Thinness for $2$-towers}\label{secThinnessGeneral}

\begin{definition}
\begin{itemize}
\item Given a rational prime number $p$, a \emph{$p$-tower} is a sequence (which may be finite) $(F_n)_{n\ge0}$ of subfields of an algebraic closure $\tilde\Q$ of $\Q$ such that for each $n$ we have $[F_{n+1}\colon F_n]=p$.
\item We say that a field extension $F/F_0$ has a \emph{$p$-tower representation} if it is the union of the fields of a $p$-tower starting from $F_0$. 
\item We call a $p$-tower $(F_n)_{n\ge0}$ \emph{thin} if the only finite extensions of $F_0$ which are subfields of $\cup_{n}F_n$ are the $F_n$. 
\end{itemize}
\end{definition}

\begin{remark}\label{remgeneralthin}
Note that if a tower $(F_n)_{n\ge0}$ is thin, then the subfields of $F=\cup_{n}F_n$ which contain $F_0$ are the $F_n$ and $F$. Indeed, for an infinite tower, if $L$ is any such subfield, it is generated over $F_0$ by a family $(x_m)_{m\ge0}$ of elements of $F$. Since for each $m$, $F_0(x_0,\dots,x_m)$ is a finite extension of $F_0$ and the tower is thin, it is equal to some $F_{k_m}$. If $L$ has infinite degree over $F_0$, then the sequence $(F_{k_m})_{m\ge0}$ is increasing, hence $L=\cup_{m\ge0}F_{k_m}=\cup_{n\ge0}F_{n}=F$. So $F$ is the only intermediate field which has infinite degree over $F_0$. 
\end{remark}

Note that if we only know that \emph{for some $k>0$}, the tower $(F_n)_{n\ge k}$ is thin, and $L$ is a subfield of $F=\cup_{n\ge0}F_n$ which contains $F_0$, then $L$ may have infinite degree over $F_0$ and be different from $F$ (as in Item 3 of Theorem \ref{main}). 

We will use the following well known lemma in many occasions. 

\begin{lemma}\label{Garling}
Let $F/K_1$ and $F/K_2$ be finite Galois field extensions with Galois group $G_1$ and $G_2$ respectively. The extension $F/K_1\cap K_2$ is a finite Galois extension if and only if the group $G$ generated by $G_1$ and $G_2$ is a finite group. If this is the case, then $G$ is isomorphic to $\Gal(G/K_1\cap K_2)$. 
\end{lemma}
\begin{proof}
See for instance \cite[Ch. 11, ex. 11.10, p. 98]{Ga86}).
\end{proof}

Since we will use this lemma only for $F$ being a number field, the hypothesis on the finiteness of $G$ is automatically satisfied. 

The following lemma is inspired by a private communication with P. Gillibert and G. Ranieri. We provide our own proof based on their idea. 

\begin{theorem}\label{generictower}
Let $(F_n)_{n\ge0}$ be a $2$-tower. Assume that $F_0$ is a number field. The tower $(F_n)_{n\ge0}$ is not thin if and only if there exists $n\ge2$ such that $F_{n}/F_{n-2}$ is Galois and its Galois group is the Klein group. 
\end{theorem}
\begin{proof}
From right to left this is obvious. We prove the other direction. Write 
$$
\Fcal_{<n}=\{F_0\subsetneq K\subsetneq F_n\colon K\ne F_i \textrm{ for every $i$}\}.
$$ 
Assume that there exists an intermediate field different from the $F_i$. Choose $n$ minimal such that $\Fcal_{<n}$ is non empty (so $n\ge2$). Choose $K$ maximal in $\Fcal_{<n}$. 

By minimality of $n$, on the one hand, there exists $j<n$ such that $K\cap F_{n-1}=F_j$, and on the other hand, $K\nsubseteq F_{n-1}$, hence $F_{n-1}\subsetneq KF_{n-1}\subseteq F_n$, hence $KF_{n-1}=F_n$ because $[F_n:F_{n-1}]=2$. In particular we have $F_{n-1}\nsubseteq K$, which implies $j\le n-2$, as otherwise we would have $j=n-1$, hence $F_{n-1}=F_j=K\cap F_{n-1}\subseteq K$.

We have $F_{j+1}\nsubseteq K$, as otherwise, since $j+1\le n-1$, we would have $F_{j+1}=F_{j+1}\cap F_{n-1}\subseteq K\cap F_{n-1}=F_j$, which is absurd. Therefore, we have $K\subsetneq KF_{j+1}$. If $KF_{j+1}\subsetneq F_n$, then $KF_{j+1}\in\Fcal_{<n}$, so by maximality of $K$, we obtain a contadiction. So we have $KF_{j+1}=F_n$. 

Since $F_{j+1}/F_j$ has degree $2$, we have 
\begin{multline}\notag
2[K:F_j]=[F_{j+1}:F_j][K:F_j]\ge [KF_{j+1}:F_j]=
[KF_{j+1}:K][K:F_{j}]=[F_n:K][K:F_j],
\end{multline}
so $[F_n:K]=2$. Since $F_n/K$ and $F_n/F_{n-1}$ are Galois, also $F_n/K\cap F_{n-1}=F_n/F_j$ is Galois by Lemma \ref{Garling}. 

Let $G$ be the Galois group of $F_n/F_j$. Let $\tau_1$ be a generator of $\Aut(F_n/K)$ and $\tau_2$ be a generator of $\Aut(F_n/F_{n-1})$ (so both $\tau_1$ and $\tau_2$ have order $2$). Since the Galois group $G$ of $F_n/F_j$ is a $2$-group, it has a non-trivial center, which by Cauchy's lemma has an element $\sigma$ of order $2$. Let $i$ be such that $\sigma\ne\tau_i$. Let $V$ be the group $\langle\tau_1,\tau_2\rangle=\langle\sigma,\tau_i\rangle$ if $\sigma$ is any of $\tau_1$ or $\tau_2$, and $V=\langle\sigma,\tau_2\rangle$ otherwise. Since $\sigma$ lies in the center of $G$, the order of both $\sigma\tau_1$ and $\sigma\tau_2$ is $2$, so $V$ is the Klein group. Let $L$ be the fixed field of $V$, so that $[F_{n}\colon L]=4$. Since in all three cases we have $\langle\tau_2\rangle\subseteq V$, by Galois correspondence $L$ is a subfield of $F_{n-1}$ and $[F_{n-1}:L]=2$. By minimality of $n$, $L$ is indeed $F_{n-2}$.
\end{proof}

It is clear that Theorem \ref{generictower} can be formulated in more general situations using the full power of Lemma \ref{Garling}. 

\begin{question}\label{questotherprimes}
Is there an analogue of Theorem \ref{generictower} for odd primes? 
\end{question}

\begin{lemma}\label{genericsubtower2} 
Let $\ell\ge2$ be an integer. Let $(F_n)_{n\ge0}$ be a $2$-tower and $F=\cup_{n\ge0} F_n$. If the tower $(F_n)_{n\ge 1}$ is thin and $L$ is a subfield of $F$ which contains $F_0$, then either $L$ is a subfield of $F_{\ell}$, or $[L:F_0]\ge2^{\ell}$. So if $L$ has degree $2^k$ for some $k\ge1$, then $L$ is a subfield of $F_{k+1}$. 
\end{lemma}
\begin{proof}
Assume that $L$ is not a subfield of $F_{\ell}$. If $L$ has infinite degree over $F_0$, there is nothing to prove. Assume that $L$ has finite degree over $F_0$, so also $F_1L$ has finite degree over $F_0$, hence $F_1L$ has finite degree over $F_1$. Since $(F_n)_{n\ge 1}$ is thin, there is some $j\ge 1$ such that $F_1L=F_j$. Since $L$ is not a subfield of $F_{\ell}$, we have $j\ge\ell+1$. Since
$$
2[L:F_0]=[F_1:F_0][L:F_0]\ge [F_1L:F_0]=[F_j:F_0]=2^j,
$$
we have $[L:F_0]\ge 2^{j-1}\ge 2^{\ell}$.
\end{proof}

If $F_0$ is a number field, $\Fcal=(F_n)_{n\ge0}$ is a $2$-tower, $(F_n)_{n\ge 1}$ is thin, $F=\cup_{n\ge0} F_n$, and $\ell\ge 2$ is an integer, we let 
$$
\Fcal_\ell=\{L\colon F_0\subseteq L\subseteq F\wedge [L:F_0]=2^\ell\wedge L\ne F_\ell\}. 
$$
If some $\Fcal_\ell$ is non-empty, we denote by $\ell_\Fcal$ the minimum of the set of $\ell$ such that $\Fcal_\ell$ is non-empty.

\begin{corollary}\label{lemgenericsubtower}
Let $F_0$ be a number field. Let $\Fcal=(F_n)_{n\ge0}$ be a $2$-tower and $F=\cup_{n\ge0} F_n$. Assume that the tower $(F_n)_{n\ge1}$ is thin. If some $\Fcal_\ell$ is non-empty, then $F_{\ell_\Fcal+1}/F_2$ is Galois. 
\end{corollary}
\begin{proof}
Let $\ell=\ell_\Fcal$. Let $L\in\Fcal_\ell$. The field $L\cap F_{\ell}$ is not any of the $F_n$ for $1\le n\le\ell$, as otherwise we would have $F_1\subseteq F_n= L\cap F_{\ell}\subseteq L$, hence $L$ would be one of the $F_j$ because the tower is thin from $n=1$, so $L=F_\ell$ for a degree reason, but this contradicts the fact that $L$ lies in $\Fcal_\ell$. In particular, $L\cap F_\ell$ is a proper subfield of $F_\ell$, hence it has degree $2^k$ for some $k<\ell$. If $k\ge2$, then $L\cap F_\ell\in \Fcal_k$ for $k<\ell$, contradicting the minimality of $\ell$. Hence we have $k\le 1$ and $L\cap F_{\ell}$ a subfield of $F_2$ by Lemma \ref{genericsubtower2}. Also by Lemma \ref{genericsubtower2}, $L$ is a subfield of $F_{\ell+1}$. Since the extensions $F_{\ell+1}/L$ and $F_{\ell+1}/F_\ell$ are quadratic, they are Galois, hence $F_{\ell+1}/L\cap F_\ell$ is Galois by Lemma \ref{Garling}, hence $F_{\ell+1}/F_2$ is Galois. 
\end{proof}

\section{Notation and some known facts}\label{NotFacts}

We will use the following three facts (see \cite[Thm. 2 and Thm. 3]{KW89}) from the general theory of quartic polynomials. 

\begin{theorem}\label{thmQuad1}
Let $P(X)=X^4+bX^2+d$ be a polynomial over a field $K$ of characteristic $\ne2$, with roots $\pm\alpha,\pm\beta$. The following conditions are equivalent: 
\begin{enumerate}
\item $P$ is irreducible over $K$. 
\item $\alpha^2$, $\alpha+\beta$ and $\alpha-\beta$ are not in $K$. 
\item $b^2-4d$, $-b+2\sqrt{d}$ and $-b-2\sqrt{d}$ are not squares in $K$.
\end{enumerate}
\end{theorem}

\begin{corollary}\label{corQuad1}
Let $P(X)=X^4+bX^2+d$ be a polynomial over a field $K$ of characteristic $\ne2$, with roots $\pm\alpha,\pm\beta$. If $\alpha^2$ and $\alpha\beta$ are not in $K$ then $P$ is irreducible over $K$. 
\end{corollary}
\begin{proof}
We have $(\alpha\pm\beta)^2=\alpha^2+\beta^2\pm2\alpha\beta=-b\pm2\alpha\beta$. 
\end{proof}

\begin{theorem}\label{thmQuad2}
Let $P(X)=X^4+bX^2+d$ be an irreducible polynomial over a field $K$ of characteristic $\ne2$, and let $F$ be the splitting field of $P$ over $K$. The Galois group of $F$ over $K$ is:
\begin{enumerate}
\item the Klein group $V_4$ if and only if $d\in K^2$;
\item the cyclic group $C_4$ if and only if $d(b^2-4d)\in K^2$;
\item the dihedral group $D_4$ if and only if $d\notin K^2$ and $d(b^2-4d)\notin K^2$; 
\end{enumerate}
\end{theorem}

\begin{notation}
\begin{enumerate}
\item $\nu\ge2$ and $x_0\ge0$ are integers. 
\item For $n\ge0$, $x_{n+1}=\pm\sqrt{\nu+x_n}$ (any choice of sign can be taken at every step). 
\item For $n\ge0$, $K_n=\Q(x_n)$, and $K=\cup_{n\ge0} K_n$. 
\item $\Kcal=(K_n)_{n\ge0}$ (and we will use the notation from Section \ref{secThinnessGeneral}).
\item If $\alpha$ is a totally real algebraic number, $\overline{\lvert\alpha\rvert}$ denotes the house of $\alpha$.
\item $\Omega$ is the set of pairs $(\nu,x_0)$ such that $[K_{n+1}:K_n]=2$ and $K_n$ is totally real for every $n\ge0$. 
\item $\Omega_{\rm inc}$ is the set of pairs $(\nu,x_0)$ in $\Omega$ such that the sequence $(\overline{\lvert x_n\rvert})_n$ is (strictly) increasing. 
\item $\Omega_{\rm dec}$ is the set of pairs $(\nu,x_0)$ in $\Omega$ such that the sequence $(\overline{\lvert x_n\rvert})_n$ is (strictly) decreasing. 
\item $(u_n)_{n\ge 0}=(u_{\nu,n})_{n\ge 0}$ is the sequence defined by $u_0=\nu^2-\nu$, and $u_{n+1}=u_n^2-\nu$ for every $n\ge0$. Note that $u_{2,n}=2$ for every $n$. We let the reader verify that the sequence $(u_n)_{n\ge0}$ is strictly increasing for $\nu\ge3$.  
\item $f_n=(u_{n-1}-x_0)(u_{n}-x_0)$, for $n\ge1$. 
\item In order to avoid confusion, we may use $\nu,x_0$ as a superscript: so for example we may write $x^{\nu,x_0}_n$ instead of $x_n$, etc.
\end{enumerate}
\end{notation}

Here are some facts that we will be using further on.
\begin{itemize} 
\item $\Omega=\Omega_{\rm inc}\cup\Omega_{\rm dec}$. See \cite[Thm. 1.4 and Lemma 2.2]{VV15} for this fact and the two next ones. 
\item The pairs $(\nu,x_0)$ in $\Omega$ such that $\nu>x_0^2-x_0$ are precisely the ones in $\Omega_{\rm inc}$
\item The pairs $(\nu,x_0)$ in $\Omega$ such that $\nu< x_0^2-x_0$ and $x_0<\nu^2-\nu$ are precisely the ones in $\Omega_{\rm dec}$.
\item If $\nu+x_0$ is congruent to $2$ or $3$ modulo $4$ (\cite[Prop. 2.15]{VV15}), or if $x_0=0$ and $\nu$ is not a square (\cite[Cor. 1.3]{St92} applied to $t^2-\nu$), then the tower increases at each step (but we do not know how to characterize the set of pairs for which the tower increases at each step). 
\end{itemize}

In particular, the only pairs $(2,x_0)$ in $\Omega$ are $(2,0)$ and $(2,1)$, and they lie in $\Omega_{\rm inc}$. 

Note that the condition that the tower increases at each step implies in particular that $\nu+x_0$, for $(\nu,x_0)\in\Omega$, is not a square.

\section{Quartic extensions within the tower}\label{secQuarticGen}

For $n\ge0$, the minimal polynomial of $x_{n+1}$ over $K_n$ is $X^2-(\nu+x_n)$. Recall that we have $u_{0}=\nu^2-\nu$ and
$$
u_{n}=u_{n-1}^2-\nu.
$$
For $n\ge\ell\ge0$, let $N_{n,\ell}$ denote the norm map from $K_n$ to $K_\ell$. 

\begin{lemma}\label{corNotSqGen2New00}
Let $(\nu,x_0)\in\Omega$. Let $n\ge\ell\ge0$ and $j\ge0$. We have 
$$
N_{n,\ell}(u_j-x_{n})=u_{j+n-\ell}-x_\ell.
$$
In particular, if $u_{j+n-\ell}-x_\ell$ is not a square in $K_{\ell}$, then $u_j-x_{n}$ is not a square in $K_{n}$. 
\end{lemma}
\begin{proof}
We have
$$
\begin{aligned}
N_{\ell+1,\ell}(u_{j+n-(\ell+1)}-x_{\ell+1})
&=u_{j+n-(\ell+1)}^2-x_{\ell+1}^2\\
&=u_{j+n-(\ell+1)}^2-(\nu+x_\ell)\\
&=\left(u_{j+n-\ell}+\nu\right)- \nu- x_{\ell}\\
&=u_{j+n-\ell}- x_{\ell},
\end{aligned}
$$
so we have the desired formula for the norm from $K_n$ to $K_\ell$ by repeating $n-\ell$ times this process. 
\end{proof}

The minimal polynomial of $x_{n+2}$ over $K_n$ is 
$$
X^4-2\nu X^2+\nu^2-\nu-x_n=
X^4-2\nu X^2+u_0-x_n. 
$$

\begin{corollary}\label{corNotSqNew00}
Let $(\nu,x_0)\in\Omega$. Let $n\ge\ell\ge0$. If $u_{n-\ell}-x_\ell$ is not a square in $K_\ell$, then the Galois group of the splitting field of $K_{n+2}/K_{n}$ is not the Klein group. 
\end{corollary}
\begin{proof}
By Lemma \ref{corNotSqGen2New00}, $u_0-x_{n}$ is not a square in $K_{n}$. We now apply Theorem \ref{thmQuad2} to the minimal polynomial of $x_{n+2}$ over $K_{n}$: in our case, $b=-2\nu$ and $d=u_0-x_{n}$, so $d$ is not a square in $K_{n}$, hence the Galois group of the splitting field of $K_{n+2}/K_{n}$ is not the Klein group. 
\end{proof}

Thus Theorem \ref{generictower} gives:

\begin{corollary}\label{corNotSq01}
Let $(\nu,x_0)\in\Omega$. If for all $n\ge0$, $u_{n}-x_0$ is not a square in $\Q$, then the tower $(K_n)_{n\ge0}$ is thin. 
\end{corollary}

\begin{lemma}\label{lemnux0b}
Let $(\nu,x_0)\in\Omega$. Unless $(\nu,x_0)=(2,1)$, we have $2\nu^2-3\nu>x_0+1$. 
\end{lemma}
\begin{proof}
This is clearly true for $(\nu,x_0)=(2,0)$ (and clearly not true for $(\nu,x_0)=(2,1)$). Assume $\nu\ge3$. The statement is then trivial when $x_0\le7$. Assume $x_0\ge8$. If $(\nu,x_0)\in\Omega_{\rm inc}$, then $\nu>x_0^2-x_0$, hence $2\nu-3>2x_0^2-2x_0-3$, hence 
$$
2\nu^2-3\nu=\nu(2\nu-3)>(x_0^2-x_0)(2x_0^2-2x_0-3)
=2x_0^4-4x_0^3-x_0^2+3x_0>x_0+1.
$$

Assume $(\nu,x_0)\in\Omega_{\rm dec}$, so that $\nu<x_0^2-x_0$ and $x_0\le\nu^2-\nu$. Assume $x_0>\nu+1$. We have $2\nu^2-2\nu\ge 2x_0$, hence $2\nu^2-3\nu\ge 2x_0-\nu>x_0+\nu+1-\nu=x_0+1$. If $x_0\le\nu+1$, then $x_0+1\le\nu+2<2\nu^2-3\nu$, because $\nu\ge3$.
\end{proof}

\begin{lemma}\label{lemsequn}
Let $(\nu,x_0)\in\Omega$. Unless $(\nu,x_0)=(2,1)$, for every $n\ge 1$, $u_n-x_0$ is not a square in $\Q$. 
\end{lemma}
\begin{proof}
This is clearly true for $(\nu,x_0)=(2,0)$ and not true for $(\nu,x_0)=(2,1)$, because $(u_n)$ is the constant sequence $(2)$. Assume $\nu\ge3$. If $u_{n+1}-x_0=z^2$ is a rational square for some $n\ge0$, by definition of $(u_n)$ we have then
$$
\nu+x_0=u_n^2-(u_{n+1}-x_0)=u_n^2-z^2=(u_n-z)(u_n+z).
$$ 
Choosing $z>0$ (recall that we assume that $\nu+x_0$ is not a rational square, so $z\ne0$), since $u_n+z>0$ and $\nu+x_0>0$, also $u_n-z>0$, so $(u_n-z)(u_n+z)\ge (u_n-z)+(u_n+z)-1$, so we have:
$$
\nu+x_0\ge(u_n-z)+(u_n+z)-1= 2u_n-1,
$$
hence $\nu+x_0\ge2u_0-1=2(\nu^2-\nu)-1$ (because $u_n$ is an increasing sequence). So we have 
$$
2\nu^2-3\nu\le x_0+1,
$$
which contradicts Lemma \ref{lemnux0b}.
\end{proof}

\begin{lemma}\label{lemsequnnu2x01}
For any $n\ge1$, $u_{2,n-1}-x^{2,1}_1=2-x^{2,1}_1$ is not a square in $K^{2,1}_1$. 
\end{lemma}
\begin{proof}
Note that $x_1=\sqrt{\nu+x_0}=\sqrt{3}$. Assume $2-x_1=2-\sqrt{3}=w^2$ for some $w\in \Q(\sqrt3)$. Since $2-\sqrt3$ is an integer, also $w$ is an integer, so we can write $w=a+b\sqrt3$ with $a,b\in\Z$. So we have $2=a^2+3b^2$, which is impossible. 
\end{proof}

\begin{theorem}\label{theoremthin}
Let $(\nu,x_0)\in\Omega$. 
\begin{enumerate}
\item The tower $(K_n^{\nu,x_0})_{n\ge1}$ is thin. 
\item The tower $(K_n^{\nu,x_0})_{n\ge0}$ is thin if and only if $u_0-x_0$ is not a square. 
\end{enumerate}
\end{theorem}
\begin{proof}
By Lemma \ref{lemsequn}, the hypothesis of Corollary \ref{corNotSqNew00} (with $\ell=0$) is satisfied for $n\ge1$ whenever $(\nu,x_0)\ne(2,1)$. For $(\nu,x_0)=(2,1)$, by Lemma \ref{lemsequnnu2x01}, the hypothesis of Corollary \ref{corNotSqNew00} (with $\ell=1$) is satisfied for $n\ge1$. So in all cases, for $n\ge1$, the Galois group of the splitting field of $K_{n+2}/K_n$ is not the Klein group, hence the tower $(K_n^{\nu,x_0})_{n\ge1}$ is thin by Theorem \ref{generictower}. 

If $u_0-x_0$ is not a square in $\Q$ (so $(\nu,x_0)\ne(2,1)$), then for every $n\ge0$ the Galois group of the splitting field of $K_{n+2}/K_n$ is not the Klein group, hence the tower $(K_n^{\nu,x_0})_{n\ge0}$ is thin, again by Theorem \ref{generictower}. 

If $u_0-x_0=a^2$ is a square, then 
$$
x_2\sqrt{\nu-x_1}=\sqrt{\nu+x_1}\sqrt{\nu-x_1}=\sqrt{\nu^2-x_1^2}=|a|\in\Q,
$$ 
so $K_2/K_0$ is Galois. The minimal polynomial of $x_2$ is $X^4-2\nu X^2+a^2$. This polynomial is indeed irreducible because we have assumed that the tower is increasing at each step, hence the Galois group is the Klein group by Theorem \ref{thmQuad2}. We deduce that there are three distinct intermediate fields strictly between $\Q$ and $\Q(x_ 2)$, hence the tower $(K_n^{\nu,x_0})_{n\ge0}$ is not thin. 
\end{proof}

By Remark \ref{remgeneralthin} and the item 2 of Theorem \ref{theoremthin}, we deduce item 1 of Theorem \ref{main}. 

When $u_0-x_0$ is a square, we need to understand better the quartic extensions within the tower. 

\begin{lemma}\label{corNotSqGen1}
For every $n\ge 1$, we have $N_{n,n-1}(\nu+x_{n})=u_0-x_{n-1}$.
\end{lemma}
\begin{proof}
We have $N_{n,n-1}(\nu+x_{n})=\nu^2-x_{n}^2
=\nu^2-(\nu+x_{n-1})
=u_0-x_{n-1}$.
\end{proof}

\begin{corollary}\label{corNotSq}
Let $(\nu,x_0)\in\Omega$. For each $n\ge1$, if $f_{n}$ is not a square in $\Q$, then the Galois group of the splitting field of $K_{n+2}/K_{n}$ is $D_4$. 
\end{corollary}
\begin{proof}
Note that $f_{n}=(u_{n-1}-x_0)(u_{n}-x_0)$ is $N_{n-1,0}((u_0-x_{n-1})(u_1-x_{n-1}))$ by Lemma \ref{corNotSqGen2New00}, and $(u_0-x_{n-1})(u_1-x_{n-1})$ is $N_{n,n-1}((\nu+x_{n})(u_0-x_{n}))$ by Lemmas \ref{corNotSqGen2New00} and \ref{corNotSqGen1}. So our hypothesis implies that $(\nu+x_{n})(u_0-x_{n})$ is not a square in $K_{n}$. 

We now apply Theorem \ref{thmQuad2} to the minimal polynomial of $x_{n+2}$ over $K_{n}$: in our case, $b=-2\nu$ and $d=u_0-x_{n}$. We have
$$
b^2-4d=4\nu^2-4(u_0-x_{n})=4\nu^2-4(\nu^2-\nu- x_{n})=4(\nu+x_{n}),
$$
hence $d(b^2-4d)$ is not a square in $K_{n}$. Therefore, the Galois group of the splitting field of $K_{n+2}/K_{n}$ is not $C_4$. Since $u_{n}-x_0$ is not a square in $\Q$ by Lemma \ref{lemsequn}, by Corollary \ref{corNotSqNew00} this Galois group is not the Klein group, hence the conclusion by Theorem \ref{thmQuad2}. 
\end{proof}

We finish by a general lemma that will be used to prove Theorem \ref{lemlatticestructure} and Proposition \ref{CorKnux020}. 

\begin{lemma}\label{lemEllCurves}
Unless $\nu=2$, the $f_n$ can be a square for at most finitely many $n$. 
\end{lemma}
\begin{proof}
For $n\ge1$, we have
$$
(u_{n-1}-x_0)(u_{n}-x_0)=(u_{n-1}-x_0)(u_{n-1}^2-\nu-x_0).
$$
Since the polynomial 
$$
P(X)=(X-x_0)(X^2-(\nu+x_0))=(X-x_0)(X-x_1)(X+x_1)
$$ 
has three distinct roots, $Y^2=P(X)$ defines an elliptic curve, so it has finitely many integral points. The quantity $(u_{n-1}-x_0)(u_{n}-x_0)$ is a square if and only if $u_{n-1}$ is the $X$-coordinate of such a point. If $\nu\ne2$, then the sequence $(u_n)_{n\ge0}$ is strictly increasing, so each possible $X$-coordinate of an integral point corresponds to exactly one value of $n$. 
\end{proof}

\section{Towers that are thin from $1$ but not from $0$}\label{secOmega1}

We denote by $\Omega^1$ the set of pairs $(\nu,x_0)$ in $\Omega$ such that the tower $(K_n)_{n\ge0}$ is not thin from $0$ (but it is thin from $1$ by Theorem \ref{theoremthin}). 

By Theorem \ref{generictower}, there exists $n\ge0$ such that the Galois group of the splitting field of $K_{n+2}/K_n$ is $V_4$, and for all $n\ge1$, the Galois group of the splitting field of $K_{n+2}/K_n$ is either $C_4$ or $D_4$. Therefore, the Galois group of the splitting field of $K_{2}/K_0$ is $V_4$. By Corollary \ref{corNotSqNew00}, this implies that $u_{0}-x_0$ is a square. Let $a$ be the non-negative integer such that $u_{0}-x_0=a^2$.

\begin{lemma}\label{lemothersqrtn}
Let $(\nu,x_0)\in\Omega^1$. The quadratic extensions of $\Q$ that lie in $K$ are 
$$
K_1=\Q(\sqrt{\nu+x_0})=\Q(\sqrt{(\nu-a)(\nu+a)}),\qquad \Q(\sqrt{2(\nu+a)}),\quad\textrm{and}\quad \Q(\sqrt{2(\nu-a)}).
$$ 
These three fields are distinct. 
\end{lemma}
\begin{proof}
First note that by Lemma \ref{genericsubtower2}, the quadratic extensions that lie in $K$ actually lie in $K_2$. Since the minimal polynomial of $x_2$ over $\Q$ is $X^4-2\nu X^2+u_0-x_0$, Theorem \ref{thmQuad2} implies that $K_2/\Q$ is Galois with Galois group the Klein group. Hence there are exactly three distinct intermediate quadratic extensions of $\Q$ (as already seen in the proof of Theorem \ref{theoremthin}). 

Consider the two roots $\alpha=x_2=\sqrt{\nu+ x_1}$ and $\beta=\sqrt{\nu- x_1}$ of the minimal polynomial of $x_2$ over $\Q$. Note that $\beta$ lies in $K_2$. We have
$$
(\alpha\pm\beta)^2=2\nu\pm 2\sqrt{\nu^2- x_1^2}=2(\nu\pm \sqrt{\nu^2-\nu-x_0})=2(\nu\pm a), 
$$
hence the fields mentioned are in $K_2$. 

If we would have $\Q(\sqrt{2(\nu+a)})=\Q(\sqrt{2(\nu-a)})$, then both $\alpha+\beta$ and $\alpha-\beta$ would lie in $\Q(\sqrt{2(\nu+a)})$, hence also $\alpha$ and $\beta$ would, but they have degree $4$ over $\Q$ by assumption. Also, if $\Q(\sqrt{(\nu-a)(\nu+a)})$ would be equal to $\Q(\sqrt{2(\nu+a)})$, then $\sqrt{2(\nu-a)}$ would lie in $\Q(\sqrt{2(\nu+a)})$, which would be a contradiction. 
\end{proof}

Note that Lemma \ref{lemothersqrtn} requires that the tower increases in the two first steps: for $\nu=3$ and $x_0=5$ we have $a=1$, but $\Q(x_2)$ has degree $2$ over $\Q$ and the three fields collapse into $\Q(\sqrt2)$ (two of them) and $\Q$ (the third one). 

For $\varepsilon=\pm1$, write 
$$
y_\varepsilon=\sqrt{2(\nu+\varepsilon a)}.
$$

\begin{lemma}\label{lemlatticestructureNotKleinb}
Let $(\nu,x_0)\in\Omega^1$, distinct from $(2,1)$. The extension $K^{\nu,x_0}_3/\Q(y_\varepsilon)$ is not Klein. 
\end{lemma}
\begin{proof}
One can easily check that the minimal polynomial of $x_2$ over $\Q(y_\varepsilon)$ is 
$$
X^2-y_\varepsilon X+\varepsilon a. 
$$
Since $x_2=x_3^2-\nu$, the minimal polynomial of $x_3$ over $\Q(y_\varepsilon)$ is 
$$
(X^2-\nu)^2-y_\varepsilon(X^2-\nu)+\varepsilon a, 
$$
which simplifies to
$$
X^4-(y_\varepsilon+2\nu)X^2+\nu y_\varepsilon+\nu^2+\varepsilon a.
$$
We apply Theorem \ref{thmQuad2} with $d=\nu y_\varepsilon+\nu^2+\varepsilon a$, and we want to know when $d$ is a square in $\Q(y_\varepsilon)$. A square in $\Q(y_\varepsilon)$ has the form 
$$
(v+wy_\varepsilon)^2=v^2+2w^2(\nu+\varepsilon a)+2vwy_\varepsilon 
$$
for some rational numbers $v$ and $w$. Consider the following system: 
$$
\begin{cases}
2vw=\nu\\
v^2+2w^2(\nu+\varepsilon a)=\nu^2+\varepsilon a
\end{cases}
$$
For the pair $(2,1)$, the system has the solution $(v,w)=(1,1)$ (for $\varepsilon=-1$ and $a=1$), so the extension is Galois with Galois group the Klein group. Replacing $v$ in the second equation, we get 
$$
\left(\frac{\nu}{2w}\right)^2+2w^2(\nu+\varepsilon a)-\nu^2-\varepsilon a=0,
$$
hence 
$$
8(\nu+\varepsilon a)w^4-4(\nu^2+\varepsilon a)w^2+\nu^2=0
$$
that we see as a polynomial in $w^2$. Its discriminant is 
$$
16(\nu^2+\varepsilon a)^2-32(\nu+\varepsilon a)\nu^2=16(a^2+\nu^4-2\nu^3).
$$
We now show that the discriminant is not a square if $(\nu,x_0)\ne(2,1)$ (so there is no rational solution $w^2$, so the above system has no solution). We have 
$$
\nu^4-2\nu^3+a^2=(\nu^2-\nu)^2-\nu^2+a^2<(\nu^2-\nu)^2. 
$$ 
So it is enough to prove that $(\nu^2-\nu-1)^2<(\nu^2-\nu)^2-\nu^2+a^2$, namely, that $-2(\nu^2-\nu)+1<-\nu^2+a^2$, or equivalentlly that $\nu^2-2\nu+a^2-1>0$. But $\nu^2-2\nu+a^2-1\ge\nu^2-2\nu=\nu(\nu-2)>0$ is true for our choice of $\nu$. So the system has no solution in the rationals, except for the pair $(2,1)$. 
\end{proof}

We can now finish the proof of item 2 of Theorem \ref{main}. 

\begin{theorem}\label{lemlatticestructure}
Let $(\nu,x_0)\in\Omega^1$ distinct from $(2,1)$. 
\begin{enumerate}
\item For every $n\ge2$, $K_n$ is the unique subfield of $K$ of degree $2^{n}$ over $\Q$. 
\item The field $K$ has no proper subfield of infinite degree over $\Q$. 
\end{enumerate}
\end{theorem}
\begin{proof}
First note that by Lemma \ref{lemEllCurves}, there exists an integer $n_0\ge1$ such that for each $n\ge n_0$, $f_{n}$ is not a square. 

By Corollary \ref{corNotSq}, the extension $K_{n+2}/K_{n}$ is not Galois when $n\ge n_0$. For the sake of contradiction, assume that for some $n\ge 2$, there is a subfield of $K$ of degree $2^{n}$ over $\Q$ which is not $K_{n}$, so that $\Kcal_n$ is non-empty (see notation in Section \ref{secThinnessGeneral}). 
By Corollary \ref{lemgenericsubtower}, the extension $K_{\ell_\Kcal+1}/K_2$ is Galois. Hence $\ell_\Kcal$ cannot be greater than or equal to $n_0+1$, as otherwise we would have $K_{\ell_\Kcal+1}/K_{\ell_\Kcal-1}$ Galois, which contradicts our hypothesis. Note that by definition $\ell_\Kcal$ is at least $2$. Let us write $\ell=\ell_\Kcal$. 

\emph{Case 1: $\ell=2$.} Let $L\in\Kcal_2$. By Lemma \ref{genericsubtower2}, $L$ is a subfield of $K_3$. By Lemma \ref{Garling}, $K_3/L\cap K_2$ is Galois. Note that $K_3/L\cap K_2$ is not of degree $2$, as we would have $L=K_2$. If $K_3/L\cap K_2$ is of degree $4$, then $L\cap K_2$ has to be $\Q(y_\varepsilon)$ (because $\Kcal_{\ge1}$ is thin), so the extension is Klein, contradicting Lemma \ref{lemlatticestructureNotKleinb}. Hence, $L\cap K_2=\Q$. Note that there cannot exist an intermediate field, say $M$, between $\Q$ and $L$, as otherwise $M$ would be one of $K_1$, $\Q(y_1)$ or $\Q(y_{-1})$, so $M$ would be a subfield of $L\cap K_2=\Q$, which is absurd. Nevertheless, since the Galois group of $K_3/\Q$ is a $2$-group, each of its subgroups of order $2$ is contained in a subgroup of order $4$ (see for instance \cite[Ch. 4, ej. 4.5, p.78]{Rotman}), so by Galois correspondence there must be an intermediate field between $\Q$ and $L$, which is a contradiction.

\emph{Case 2: $3\le\ell\le n_0$}. Let $L\in\Kcal_{\ell}$. By Lemma \ref{genericsubtower2}, $L$ is a subfield of $K_{\ell+1}$. Let $L'=L\cap K_\ell$. The degree of $L'$ over $\Q$ is at most $2$, because $\Kcal_{2}$, \dots, $\Kcal_{\ell-1}$ are empty, and if $[L'\colon\Q]=2^r$ for some $2\le r\le \ell-1$ then we would have $L'=K_{r}$, contradicting the fact that $\Kcal$ is thin from $1$. \\
\emph{Subcase 2.1: $[L'\colon \Q]=2$}. In this case we have $L'\ne K_1$, again because $\Kcal$ is thin from $1$, so $L'$ is either $\Q(y_1)$ or $\Q(y_{-1})$. There cannot exist an intermediate field $M$ between $L'$ and $L$, as otherwise we would have $M=K_j$ for some $2\le j<\ell$ (because $\Kcal_{2}$, \dots, $\Kcal_{\ell-1}$ are empty), so $M=K_j$ would be a subfield of $L$, contradicting the fact that $\Kcal$ is thin from $1$. We conclude as in the previous case. \\
\emph{Subcase 2.2: $L'=\Q$}. Again, there cannot be an intermediate field $M$ between $L$ and $\Q$, as otherwise, we would have either $M=K_2$ (again because $\Kcal_2$ is empty) which is absurd (again because $M$ would be a subfield of $L'=\Q$), or $M$ has degree $2$, so $M$ is either $K_1$ or one of the $\Q(y_\varepsilon)$, so $M$ would be a subfield of $K_3$, so of $L'=\Q$. We conclude as in the previous cases.

Finally, assume that there exists a subfield $L$ of $K$ which has infinite degree over $\Q$. Let $\alpha\in L\setminus K_2$ of degree $2^r$. Note that $r$ cannot be $2$ because $\Kcal_2$ is empty, and it cannot be $1$ because $\alpha\notin K_2$. So from item 1, we have $\Q(\alpha)=K_r$. Therefore, $L\supseteq K_r\supseteq K_1$, and since the tower $(K_r)_{r\ge1}$ is thin, we conclude with Remark \ref{remgeneralthin} that $L$ is equal to $K$. 
\end{proof}

We finish this section with a few results on $\Omega^1$ that will be useful in the next sections. Let us write
$$
\Sigma=\{(\nu,u_0-a^2)\in\Omega\colon 1\le a\le\nu-1\}. 
$$

\begin{lemma}\label{lemcarasquare}
We have $\Omega^1=\Sigma$. 
\end{lemma}
\begin{proof}
If $(\nu,x_0)=(\nu,u_0-a^2)\in\Sigma$, then $u_0-x_0=\nu^2-\nu-(\nu^2-\nu-a^2)=a^2$ is a square. Assume that we have $u_0-x_0=a^2$ for some integer $a$, say non-negative, so in particular $x_0$ has the required shape. Note that $a$ cannot be $0$, since we assumed that $\nu+x_0$ is not a square. Since $x_0\ge0$, we have $u_0-a^2\ge0$, hence $a^2\le \nu^2-\nu<\nu^2$, hence $a<\nu$. 
\end{proof}

The next lemma is part of item 2 of Theorem \ref{main}. 

\begin{lemma}\label{lemsigmainfinite}
The set $\Omega^1$ is infinite. 
\end{lemma}
\begin{proof}
We know that if $\nu+x_0$ is congruent to $2$ or $3$ modulo $4$, then the tower increases at each step. In $\Omega^1=\Sigma$, we have $\nu+x_0=\nu^2-a^2$ for some $1\le a\le \nu-1$. Write $a=\nu-k$, so that $\nu+x_0=2\nu k-k^2$ for some $1\le k\le \nu-1$. For $\nu$ congruent to $0$ or $2$ modulo $4$, we can choose $k$ congruent to $1$ or $3$ modulo $4$. 
\end{proof}

\begin{lemma}\label{corcarasquarex0}
 If $(\nu,x_0)\in\Omega^1$ then $\nu-1\le x_0\le\nu^2-\nu-1$.
\end{lemma}
\begin{proof}
By Lemma \ref{lemcarasquare}, $x_0=\nu^2-\nu-a^2$ for some $a$ such that $1\le a\le\nu-1$. So we have $-(\nu-1)^2\le -a^2\le -1$, hence $\nu-1\le \nu^2-\nu-a^2\le \nu^2-\nu-1$. 
\end{proof}

\begin{lemma}\label{corcarasquareDEC}
In $\Omega^1$, $(2,1)$ and $(3,2)$ are the only pairs with $\nu\le3$, and they lie in $\Omega_{\rm inc}$. All the pairs $(\nu,x_0)\in\Omega^1$ with $\nu\ge4$ lie in $\Omega_{\rm dec}$. 
\end{lemma}
\begin{proof}
Let $(\nu,\nu^2-\nu-a^2)\in\Omega^1$. If $\nu=2$, then $1\le a\le \nu-1$ gives $a=1$, hence $x_0=1$ and $x_0^2-x_0=0<\nu$. If $\nu=3$, then $a=1$ or $a=2$, hence $x_0=5$ or $x_0=2$. For $x_0=5$, we have already seen that the tower does not increase at step $2$. So $x_0=2$, and we have $x_0^2-x_0=2<3=\nu$. So, in both cases the tower is totally real and we are in the increasing case. 

Assume $\nu\ge 4$. By Lemma \ref{corcarasquarex0}, we have $x_0<\nu^2-\nu$, and $x_0\ge\nu-1\ge3$ gives $\nu\le x_0+1<x_0^2-x_0$, so the tower is totally real and we are in the decreasing case. 
\end{proof}

\section{Case $\nu=2$}\label{subsecnu2}

In this section, we prove Item 3 of Theorem \ref{main} (by Lemma \ref{corcarasquareDEC}, the only pair in $\Omega^1$ with $\nu=2$ is the pair $(2,1)$). 

\begin{lemma}
$K^{2,0}_{n}$ is a proper subfield of $K^{2,1}_{n+1}$ for every $n\ge0$ (hence $K^{2,0}$ is a subfield of $K^{2,1}$). Moreover, $K^{2,1}$ is an infinite cyclotomic extension of $\Q$. 
\end{lemma}
\begin{proof}
For any given $m\ge2$, let us denote by $\zeta_m$ a primitive $m$-th root of unity in $\C$. Note that we have $\overline{\lvert x_1^{2,0}\rvert}=\sqrt{2}=2\cos(\pi/4)$ and $\overline{\lvert x_1^{2,1}\rvert}=\sqrt{3}=2\cos(\pi/6)$, and more generally
$$
\overline{\lvert x_n^{2,0}\rvert}=2\cos\left(\frac{2\pi}{2^{n+2}}\right)\qquad 
\overline{\lvert x_{n+1}^{2,1}\rvert}=2\cos\left(\frac{2\pi}{3\times2^{n+2}}\right) 
$$
for every $n\ge0$. We have $\zeta_{2^{n+2}}=\zeta_{3\times2^{n+2}}^3\in\Q(\zeta_{3\times2^{n+2}})$, hence 
$$
\overline{\lvert x_n^{2,0}\rvert}=\zeta_{2^{n+2}}+\overline{\zeta_{2^{n+2}}}
$$
lies in the totally real part of $\Q(\zeta_{3\times2^{n+2}})$, which is $\Q\left(\overline{\lvert x_{n+1}^{2,1}\rvert}\right)$, and since $\Q(\zeta_{3\times2^{n+2}})$ is an abelian extension of $\Q$, its totally real part is Galois, so we have $\Q\left(\overline{\lvert x_{n+1}^{2,1}\rvert}\right)=\Q(x_{n+1}^{2,1})=K^{2,1}_{n+1}$. 

The fields $K^{2,0}_{n}$ and $K^{2,1}_{n+1}$ are distinct because they don't have the same degree over $\Q$. 
\end{proof}

We now list some easy facts. Recall that we defined $M_n=\Q(\sqrt3 x^{2,0}_n)$. 

\begin{lemma} We have: 
\begin{enumerate}
\item $\sqrt{3}\notin K^{2,0}$. 
\item $K^{2,1}_{n+1}=K^{2,0}_n(\sqrt{3})$ for every $n\ge0$. 
\item $K^{2,0}\ne K^{2,1}$.
\item $K^{2,1}=K^{2,0}(\sqrt{3})$. 
\item For each $n\ge0$, the field $M_{n+1}$ is strictly between $K^{2,0}_n$ and $K^{2,1}_{n+2}$. 
\item For each $n\ge0$, the fields $K^{2,0}_{n+1}$, $K^{2,1}_{n+1}$ and $M_{n+1}$ are distinct. 
\end{enumerate}
\end{lemma}
\begin{proof}
\begin{enumerate}
\item If not, since $(K^{2,0}_n)_{n\ge0}$ is thin, we would have $\Q(\sqrt3)=K^{2,0}_1=\Q(\sqrt2)$. 
\item This is immediate by the previous item and the fact that the degree of $K^{2,1}_{n+1}$ over $\Q$ is $2^{n+1}$ while the degree of $K^{2,0}_{n}$ over $\Q$ is $2^n$. 
\item This is because $\sqrt{3}$ lies in $K^{2,1}_1=\Q(\sqrt{3})$ but not in $K^{2,0}$.
\item The tower $(K^{2,1}_n)_{n\ge1}$ is thin, hence, since $K^{2,0}(\sqrt3)\subseteq K^{2,1}$ contains $\Q(\sqrt3)=K^{2,1}_1$ and has infinite degree, it is $K^{2,1}$ by Remark \ref{remgeneralthin}. 
\item We have $\sqrt3\in K^{2,1}_1\subseteq K^{2,1}_{n+2}$, and $x^{2,0}_{n+1}\in K^{2,1}_{n+2}$, hence $M_{n+1}$ is a subfield of $K^{2,1}_{n+2}$. Also, the square of $\sqrt3 x^{2,0}_{n+1}$ is $3(2+x^{2,0}_n)$, hence $K^{2,0}_n$ is a subfield of $M_{n+1}$. Since $\sqrt3x^{2,0}_{n+1}=\sqrt3\sqrt{2+x^{2,0}_n}=\sqrt{6+3x^{2,0}_n}$ has degree $\leq 2^{n+1}$ over $\Q$, $M_{n+1}$ is a proper subfield of $K^{2,1}_{n+2}$. If $K^{2,0}_n$ would be equal to $M_{n+1}$, we would have $\sqrt{3}x^{2,1}_{n+1}\in K^{2,0}_{n}\subseteq K^{2,0}_{n+1}$, hence $\sqrt{3}$ would lie in $K^{2,0}_{n+1}$, which is a contradiction. 
\item First note that $\sqrt{3}$ lies in $K^{2,1}_{n+1}$ but not in $K^{2,0}_{n+1}$, so these two fields are distinct. As in the previous item, if $M_{n+1}$ would be equal to $K^{2,0}_{n+1}$, then $\sqrt{3}$ would lie in $K^{2,0}_{n+1}$, which is not the case. Finally, if $M_{n+1}$ would be equal to $K^{2,1}_{n+1}$, then we would have $x^{2,0}_{n+1}$ in $K^{2,1}_{n+1}$ (because $\sqrt{3}\in K^{2,1}_{n+1}$), hence $K^{2,0}_{n+1}\subseteq K^{2,1}_{n+1}$, so these two fields would be equal because they have the same degree over $\Q$. 
\end{enumerate}
\end{proof}

\begin{lemma} 
The extension $K^{2,1}_{n+2}/\Q$ is Galois with Galois group $C_{2^{n+1}}\times C_2$
\end{lemma}
\begin{proof}
Since $K^{2,1}$ is cyclotomic, every extension $K^{2,1}_{n+2}/\Q$ is a degree $2^{n+2}$ abelian Galois extension of $\Q$. Hence its Galois group $G_{n+2}$ is $C_{2^{\ell_m}}\times\dots\times C_{2^{\ell_1}}$ for some $\ell_j$ such that $\ell_1+\dots+\ell_m=n+2$. Assume $\ell_m\geq\dots\geq\ell_1\geq1$. Since $K^{2,1}_{n+2}$ has exactly three distinct subfields of degree $2$ over $\Q$ by Lemma \ref{lemothersqrtn}, by Galois correspondence, $G_{n+2}$ has exactly three subgroups of index $2$.  In particular it is not cyclic, so $m\ge2$. Let $a_j=2^{\ell_j-1}$ for $j=1,\dots,m$ (so each $a_j$ has order two in $C_{2^{\ell_j}}$). If $m\ge3$, then $(2^{\ell_m-1},0,0,\dots,0)$, $(0,2^{\ell_{m-1}-1},0,\dots,0)$, $(0,0,2^{\ell_{m-2}-1},0,0,\dots,0)$, and $(2^{\ell_m-1},2^{\ell_{m-1}-1},0,\dots,0)$ generate four distinct subgroups of order $2$, hence there are at least $4$ distinct subgroups of index $2$ (see \cite[Ch. 10, ex. 10.54, p. 341]{Rotman}), which is a contradiction. So we have $m=2$. By Lemma \ref{genericsubtower2}, the only quartic extensions of $\Q$ that lie in $K^{2,1}$ lie in $K^{2,1}_3$. Since $m=2$, the only option for $G_3$ is $C_4\times C_2$, which has exactly three subgroups of order two, hence there are exactly three quartic extensions of $\Q$ that lie in $K^{2,1}$. If $\ell_1\ge2$, then we would have four groups of order $4$ that lie in $G_{n+2}$: $<(2^{\ell_2-2},0)>$, $<(0,2^{\ell_1-2})>$, $<(2^{\ell_2-2},2^{\ell_1-2})>$ and $<(2^{\ell_2-1},0),(0,2^{\ell_1-1})>$. Hence there would be four subgroups of index four, so by Galois correspondence, there would be four quartic extensions of $\Q$ lying in $K^{2,1}$, which is a contradiction. 
\end{proof}

\begin{lemma} 
For every $n\ge0$ and every $\ell=1,\dots,n+1$, the subfields of $K^{2,1}_{n+2}$ of degree $2^\ell$ over $\Q$ are $K^{2,0}_\ell$, $K^{2,1}_\ell$ and $M_\ell$
\end{lemma}
\begin{proof}
We let to the reader show that the group $C_{2^{n+1}}\times C_2$ has exactly three subgroups of order $2^\ell$ (hence also of order $2^{n+2-\ell}$): $<(2^{n+1-\ell},0)>$, $<(2^{n+1-\ell},1)>$ and $<(0,1),(2^{n-\ell},0)>$. 
\end{proof}

\begin{lemma} 
For each $n\ge0$, $M_{n+1}$ does not contain $M_n$
\end{lemma}
\begin{proof}
If not, then $M_{n+1}=\Q(\sqrt{3}x^{2,0}_{n+1})$ contains both $K^{2,0}_n=\Q(x^{2,0}_n)$ and $\sqrt{3}x^{2,0}_n$, hence it contains $\sqrt{3}$, hence it contains $x^{2,0}_{n+1}$. This is a contradiction because we would have $K^{2,0}_{n+1}\subseteq M_{n+1}$. 
\end{proof}

Therefore, so far, we know the lattice of subfields of $K^{2,1}$ that have finite degree over $\Q$ ---  see Figure \ref{fig:SubCamp21Intro}. 

\begin{lemma}
The only proper subfield of $K^{2,1}$ of infinite degree over $\Q$ is $K^{2,0}$.
\end{lemma}
\begin{proof}
Let $L$ be a subfield of $K^{2,1}$ of infinite degree. From the lattice of subfields of finite degree, $L$ contains infinitely many of the $K^{2,0}_n$, hence it contains $K^{2,0}$, but $K^{2,1}=K^{2,0}(\sqrt{3})$ is a degree $2$ extension, so $L$ is either $K^{2,0}$ or $K^{2,1}$. 
\end{proof}

\section{Towers with $\sqrt2$}\label{secSqrt2}

We return to one of the original motivations for this work, which was the problem of determining the pairs $(\nu,x_0)$ for which $K^{2,0}$ is a subfield of $K^{\nu,x_0}$. We prove three independent results in this direction, the last one solving precisely the latter problem.

\begin{proposition}\label{thmsqrt2}
Let $(\nu,x_0)\in\Omega$. The square root of $2$ is in $K$ if and only if either the square free part of $\nu+x_0$ is $2$, or $(\nu,x_0)$ belongs to one of the two following sets:
\begin{enumerate}
\item $\Sigma_1=\{(\nu,u_0-(\nu-k^2)^2)\in\Sigma\colon \nu\ge 2,\, 1\le k\le \sqrt{\nu-1}\}$, or
\item $\Sigma_2=\{(\nu,u_0-(\nu-k^2)^2)\in\Sigma\colon \nu\ge 3,\, \sqrt{\nu+1}\le k\le \sqrt{2\nu-1}\}$.
\end{enumerate}
\end{proposition}
\begin{proof}
Write $s$ for the square free part of $\nu+x_0$. By Lemma \ref{genericsubtower2}, we know that $\sqrt{2}\in K$ if and only if $\sqrt{2}\in K_2$. 

From left to right. Assume $\sqrt{2}\in K_2$. Either we have $K_1=\Q(\sqrt2)$, in which case $s=2$, or $\Q(\sqrt2)$ is a subfield of $K_2$ distinct from $K_1$. Note that the latter implies $s\ne2$, and by Theorem \ref{theoremthin}, since the tower is not thin, $\nu^2-\nu-x_0=a^2$ is a square in $\Z$. Therefore, by Lemma \ref{lemothersqrtn}, $\Q(\sqrt2)$ is either $\Q(\sqrt{2(\nu-a)})$ or $\Q(\sqrt{2(\nu+a)})$, hence either $\nu-a$ is a square, or $\nu+a$ is a square.

If $\nu-a=k^2$, say for some positive $k$, then $a=\nu-k^2$, and the condition $1\le a\le\nu-1$ is equivalent to $1\le k\le \sqrt{\nu-1}$. Therefore, the subset of $\Sigma$ such that $\nu-a$ is a square is $\Sigma_1$. 

If $\nu+a=k^2$, say for some positive $k$, then $a=k^2-\nu$, and the condition $1\le a\le\nu-1$ is equivalent to $\sqrt{\nu+1}\le k\le \sqrt{2\nu-1}$. Note that this cannot happen for $\nu=2$. Therefore, the subset of $\Sigma$ such that $\nu+a$ is a square is $\Sigma_2$.

From right to left. If $s=2$, then $\Q(\sqrt2)=K_1$. Assume $s\ne2$, hence by hypothesis $(\nu,x_0)$ belongs to $\Sigma_1$ or $\Sigma_2$. So in particular, we have $x_0=u_0-(\nu-k^2)^2$ for some integer $k$, hence $u_0-x_0=(\nu-k^2)^2$ is a square (so the tower is not thin). If $(\nu,x_0)\in\Sigma_1$, then $\nu-k^2>0$, hence $\Q(\sqrt{2(\nu-(\nu-k^2))})=\Q(\sqrt2)$. If $(\nu,x_0)\in\Sigma_2$, then $k^2-\nu>0$, hence $\Q(\sqrt{2(\nu+(k^2-\nu))})=\Q(\sqrt2)$. These two quadratic extensions are in $K^{\nu,x_0}$ by Lemma \ref{lemothersqrtn}.
\end{proof}

Note that in the case that $\sqrt2$ appeared in $K$ because the square-free part of $\nu+x_0$ is $2$, the tower may or may not be thin.

\begin{proposition}\label{propx1ISksqrt2}
The set of pairs $(\nu,x_0)\in\Omega$ such that the square-free part of $\nu+x_0$ is $2$ and $x^{2,0}_2$ lies in $K^{\nu,x_0}$ is exactly the set 
$$
X=\{(2(b^2+6bd+10d^2),2(b^2+8bd+14d^2)^2-2(b^2+6bd+10d^2))\in\Omega\colon b,d\in\Z\}.
$$
Moreover, we have the following identity:
\begin{equation}\label{eqpropx1ISksqrt2}
x^{\nu,x_0}_2=\sqrt{2(b^2+6bd+10d^2)+(b^2+8bd+14d^2)\sqrt2}=b\sqrt{2+\sqrt2}+d\left(\sqrt{2+\sqrt2}\right)^3.
\end{equation}
\end{proposition}
\begin{proof}
Suppose that the square-free part of $\nu+x_0$ is $2$ (hence $K^{\nu,x_0}_1=K^{2,0}_1=\Q(\sqrt{2})$), and that $x^{2,0}\in K^{\nu,x_0}$. We will show that $(\nu,x_0)$ lies in $X$. Since the tower $(K^{\nu,x_0}_n)_{n\ge1}$ is thin by Theorem \ref{theoremthin} and $K^{2,0}_2$ contains $\Q(\sqrt{2})=K^{\nu,x_0}_1$, we have $K^{\nu,x_0}_2=K^{2,0}_2$. 

Choose $\kappa\in\Z$ positive such that $\nu+x_0=2\kappa^2$. In this proof we write $y=x^{2,0}_2=\sqrt{2+\sqrt2}$ and $z=\sqrt{2-\sqrt2}$. Since $\{1,y,y^2,y^3\}$ is a power basis for  $\Q(y)$, we can write 
\begin{equation}\label{eqpowerbasisnu2x00}
x^{\nu,x_0}_2=a+by+cy^2+dy^3.
\end{equation}
The minimal polynomial of $x^{\nu,x_0}_2$ over $\Q$ is $P^{\nu,x_0}=X^4-2\nu X^2+\nu^2-2\kappa^2$. So the trace of $x^{\nu,x_0}_2$ over $\Q$ is $0$, and Equation \eqref{eqpowerbasisnu2x00} gives
$$
(a+by+cy^2+dy^3)+(a-by+cy^2-dy^3)
+(a+bz+cz^2+dz^3)+(a-bz+cz^2-dz^3)=0, 
$$
hence $4a+2c(y^2+z^2)=4a+8c=0$, and Equation \eqref{eqpowerbasisnu2x00} becomes 
\begin{equation}\label{eqpowerbasisnu2x00b}
x^{\nu,x_0}_2=-2c+by+cy^2+dy^3.
\end{equation}
Write 
$$
r_1=-2c+by+cy^2+dy^3,\qquad r_2=-2c-by+cy^2-dy^3
$$
$$
r_3=-2c+bz+cz^2+dz^3,\qquad\textrm{ and}\qquad r_4=-2c-bz+cz^2-dz^3.
$$
From Equation \eqref{eqpowerbasisnu2x00b}, the coefficients of $X$ in $P^{\nu,x_0}$ give the system
\begin{equation}\label{casespowerbasis}
\begin{cases}
r_1r_2+r_1r_3+r_1r_4+r_2r_3+r_2r_4+r_3r_4=-2\nu\\
r_1r_2r_3+r_1r_2r_4+r_1r_3r_4+r_2r_3r_4=0\\
r_1r_2r_3r_4=\nu^2-2\kappa^2. 
\end{cases}
\end{equation}
A simple computation gives 
$$
r_1r_2=4c^2-(4c^2+b^2)y^2+(c^2-2bd)y^4-d^2y^6,
$$
so we have
$$
r_1r_2r_3=[4c^2-(4c^2+b^2)y^2+(c^2-2bd)y^4-d^2y^6][-2c+bz+cz^2+dz^3],
$$
$$
r_1r_2r_4=[4c^2-(4c^2+b^2)y^2+(c^2-2bd)y^4-d^2y^6][-2c-bz+cz^2-dz^3],
$$
hence
$$
r_1r_2(r_3+r_4)=[4c^2-(4c^2+b^2)y^2+(c^2-2bd)y^4-d^2y^6](-4c+2cz^2)
$$
and by symmetry
$$
r_3r_4(r_1+r_2)=[4c^2-(4c^2+b^2)z^2+(c^2-2bd)z^4-d^2z^6](-4c+2cy^2). 
$$
Note that $(-4c+2cz^2)+(-4c+2cy^2)=-8c+2c(y^2+z^2)=0$. So the sum of the two above expressions gives
\begin{multline}\notag
-(4c^2+b^2)[y^2(-4c+2cz^2)+z^2(-4c+2cy^2)]\\
+(c^2-2bd)[y^4(-4c+2cz^2)+z^4(-4c+2cy^2)]\\
-d^2[y^6(-4c+2cz^2)+z^6(-4c+2cy^2)]=0,
\end{multline}
hence
\begin{multline}\notag
-(4c^2+b^2)(-16c+4cy^2z^2)\\
+(c^2-2bd)[-4c(y^4+z^4)+2c(y^4z^2+z^4y^2)]\\
-d^2[-4c(y^6+z^6)+2c(y^6z^2+z^6y^2)]=0,
\end{multline}
hence, since $y^2z^2=2$, $y^2+z^2=4$, $y^4+z^4=12$, $y^6+z^6=40$, we get
$$
8c(4c^2+b^2)-32c(c^2-2bd)+112cd^2=0
$$
hence
$$
8c(b^2+8bd+14d^2)=0. 
$$
The discriminant of $b^2+8bd+14d^2$ seen as a polynomial in $b$ is $64d^2-56d^2=8d^2$, so if $c\ne0$ then $b=-4d\pm d\sqrt2$, which implies $b=d=0$. 

Assume $b=d=0$. We show that this is impossible. We have 
$$
r_1=c(-2+y^2)=r_2, \qquad r_3=c(-2+z^2)=r_4,
$$
hence Equations \eqref{casespowerbasis} give
$$
\begin{aligned}
-2\nu&=r_1r_2+r_1r_3+r_1r_4+r_2r_3+r_2r_4+r_3r_4\\
&=r_1^2+2r_1r_3+2r_2r_3+r_3^2=r_1^2+4r_1r_3+r_3^2\\
&=c^2\left(4-4y^2+y^4+4(4-2(y^2+z^2)+y^2z^2)+4-4z^2+z^4\right)\\
&=c^2(24-12(y^2+z^2)+y^4+z^4+4y^2z^2)=-4c^2, 
\end{aligned}
$$
hence $\nu=2c^2$. So we have
$$
\begin{aligned}
\nu^2-2\kappa^2&=r_1r_2r_3r_4=r_1^2r_3^2\\
&=c^4(4-4y^2+y^4)(4-4z^2+z^4)\\
&=c^4(16-16z^2+4z^4-16y^2+16y^2z^2-4y^2z^4+4y^4-4y^4z^2+y^4z^4)\\
&=c^4(16-16(y^2+z^2)+4(y^4+z^4)+16y^2z^2-4y^2z^2(y^2+z^2)+y^4z^4)\\
&=c^4(16-16\times4+4\times12+16\times2-4\times2\times4+2^2)\\
&=4c^4=\nu^2,
\end{aligned}
$$
hence $2\kappa^2=\nu+x_0=0$, which is impossible.

Assume $c=0$. We have
$$
r_1=by+dy^3=-r_2, \qquad r_3=bz+dz^3=-r_4,
$$
hence Equations \eqref{casespowerbasis} give
\begin{multline}\notag
-2\nu=r_1r_2+r_1r_3+r_1r_4+r_2r_3+r_2r_4+r_3r_4=-(r_1^2+r_3^2)=\\
-[b^2(y^2+z^2)+2bd(y^4+z^4)+d^2(y^6+z^6)]=-(4b^2+24bd+40d^2), 
\end{multline}
which gives the expression that we wanted for $\nu$. Also we have
\begin{multline}\notag
\nu^2-2\kappa^2=r_1r_2r_3r_4=r_1^2r_3^2=y^2z^2(b+dy^2)^2(b+dz^2)^2=\\
2(b^2+bdz^2+bdy^2+d^2y^2z^2)^2=
2(b^2+4bd+2d^2)^2. 
\end{multline}
hence 
$$
\begin{aligned}
2\kappa^2&=(2b^2+12bd+20d^2)^2-2(b^2+4bd+2d^2)^2\\
&=2b^4+32b^3d+184b^2d^2+448bd^3+392d^4\\
&=2(b^2+8bd+14d^2)^2.
\end{aligned}
$$
The expression for $x_0$ follows. Hence $(\nu,x_0)$ lies in $X$ and Equation \eqref{eqpropx1ISksqrt2} is proven. 

Conversely, assume that $(\nu,x_0)$ lies in $X$. So in particular, the square-free part of $\nu+x_0$ is two, and a simple calculation shows that Equation \eqref{eqpropx1ISksqrt2} is satisfied, namely, 
$$
x^{\nu,x_0}_2=by+dy^3=y(b+dy^2)
$$
and we conclude because $b+dy^2\in K^{2,0}_1=\Q(\sqrt{2})=K^{\nu,x_0}_1$. 
\end{proof}

\begin{proof}[Proof of Corollary \ref{CorKnux020}]
If $(K^{\nu,x_0})_{n\ge0}$ is thin, then $K^{\nu,x_0}_n=K^{2,0}_n$ for every $n\ge0$. So in particular all the quartic extensions $K^{\nu,x_0}_{n+2}/K^{\nu,x_0}_n$ are Galois with Galois group $C_4$. If $\nu\ne2$, then by Lemma \ref{lemEllCurves}, all but finitely many of the $f_n$ are non-squares, and we deduce from Corollary \ref{corNotSq} that the Galois group of the splitting field of $K^{\nu,x_0}_{n+2}/K^{\nu,x_0}_n$ is not $C_4$ for those $n$. 

If $(\nu,x_0)\in\Omega^1$ and $(\nu,x_0)\ne(2,1)$, then by Theorem \ref{lemlatticestructure} and Lemma \ref{lemEllCurves}, we know that for every $n\ge2$, $K^{\nu,x_0}_n$ is the unique subfield of $K^{\nu,x_0}$ of degree $2^n$ over $\Q$, so we have $K^{\nu,x_0}_n=K^{2,0}_n$ for every $n\ge 2$. In particular we have $K^{\nu,x_0}=\cup_{n\ge2}K^{\nu,x_0}_n=\cup_{n\ge2}K^{2,0}_n=K^{2,0}$, so $K^{\nu,x_0}_n=K^{2,0}_n$ for every $n\ge0$ because $(K^{2,0}_n)_{n\ge0}$ is thin, so $(K^{\nu,x_0}_n)_{n\ge0}$ is thin, which contradicts the fact that $(\nu,x_0)$ lies in $\Omega^1$. So in this case we have $(\nu,x_0)=(2,1)$. \end{proof}

If $\zeta_m+\zeta_{m}^{-1}$ is in $K^{\nu,x_0}$ for some $m$, then $m$ has the form either $2^rp_1p_2$ for $r\le2$, or $2^rp_1$ for $r\ge3$, or $2^r$ for $r\ge2$, where $p_1$ and $p_2$ are distinct Fermat primes.

\begin{proof}[Proof of Corollary \ref{CorFermat}] By \cite{CVV20}, we know that $m$ has the form $2^rp_1\dots p_s$, where $r\ge2$ and the $p_i$ are distinct Fermat primes. Each Fermat prime $p_i$ contributes with the unique quadratic extension $\Q(\sqrt{p_i})$. Also $\Q(\zeta_4)$ contributes with $\Q(\sqrt{-1})$, for each $r\ge3$, $\Q(\zeta_{2^r})$ contributes with $\Q(\sqrt{-1})$, $\Q(\sqrt{2})$ and $\Q(\sqrt{-2})$. Since square roots of primes are linearly independent over $\Q$, and by Theorem \ref{main} we know that $K^{\nu,x_0}$ has at most three quadratic sub-extensions, the corollary follows. 

\end{proof}

\noindent Xavier Vidaux (corresponding author)\\
Universidad de Concepci\'on, Concepci\'on, Chile\\
Facultad de Ciencias F\'isicas y Matem\'aticas\\
Departamento de Matem\'atica\\
Casilla 160 C\\
Email: xvidaux@udec.cl\\

\noindent Carlos R. Videla\\
Mount Royal University, Calgary, Canada\\
Department of Mathematics and Computing\\
email: cvidela@mtroyal.ca


\begin{thebibliography}{}

\bibitem[Ca18]{Ca18} Castillo, Marianela, \emph{On the Julia Robinson Number of Rings of Totally Real Algebraic Integers in Some Towers of Nested Square Roots}, PhD thesis, Universidad de Concepci\'on, Chile (2018).\\
{\tt http://dmat.cfm.cl/dmat/doctorado/tesis/}


\bibitem[Ca21]{Ca21} Castillo, Marianela, \emph{A dynamical characterization for monogenity at every level of some infinite $2$-towers}. Canad. Math. Bull. (2021), pp. 1--9. \url{http://dx.doi.org/10.4153/S0008439521000874}. 

\bibitem[CVV20]{CVV20} Castillo, Marianela; Vidaux, Xavier; Videla, Carlos R., \emph{Julia Robinson numbers and arithmetical dynamic of quadratic polynomials}. Indiana Univ. Math. J. 69 (2020), no. 3, 873--885.


\bibitem[Ga86]{Ga86} Garling, D. J. H., \emph{A course in Galois theory.} Cambridge University Press, Cambridge, 1986. viii+167 pp. ISBN: 0-521-32077-1; 0-521-31249-3. 

%
\bibitem[GR19]{GR19} Gillibert, Pierre; Ranieri, Gabriele Julia Robinson numbers. Int. J. Number Theory {\bf 15} (2019), no. 8, 1565--1599.

\bibitem[KW89] {KW89}  Kappe, Luise-Charlotte; Warren, Bette, \emph{An elementary test for the Galois group of a quartic polynomial.} Amer. Math. Monthly {\bf 96} (1989), no. 2, 133--137.


\bibitem[Li22]{Li22} Li, Ruofan, \emph{On number fields towers defined by iteration of polynomials.} Arch. Math. (Basel) 119 (2022), no.4, 371--379


\bibitem[Mi14]{Mi14}  Misseldine, Andrew Frank, \emph{Algebraic and Combinatorial Properties of Schur Rings over Cyclic Groups.} Thesis (Ph.D.)--Brigham Young University. 2014. 162 pp. ISBN: 979-8662-51815-5.

\bibitem[Od85]{Od85}  R. W. K. Odoni, \emph{The Galois theory of iterates and composites of polynomials}. Proc. London Math. Soc. (3) {\bf 51}, no. 3, 385--414 (1985).

\bibitem[Ok22]{Ok22} Okazaki, Masao, \emph{Relative Northcott numbers for the weighted Weil heights}. arXiv:2206.05440 (2022). 

\bibitem[OS22]{OS22} Okazaki, Masao; Sano, Kaoru, \emph{Northcott numbers for generalized weighted Weil heights.}
arXiv:2308.03981 (2022). 

\bibitem[OS23]{OS23} Okazaki, Masao; Sano, Kaoru, \emph{Northcott numbers for the weighted Weil heights.}
Atti Accad. Naz. Lincei Rend. Lincei Mat. Appl. 34 (2023), no.1, 127--144.

\bibitem[PTW22]{PTW22} Pazuki, Fabien; Technau, Niclas; Widmer, Martin, \emph{Northcott numbers for the house and the Weil height}. Bull. Lond. Math. Soc. 54 (2022), no. 5, 1873--1897. 


\bibitem[Rob59]{Rob59} Robinson, J., \emph{Problems of number theory arising in metamathematics}, Report of the Institute in the Theory of Numbers (Boulder) (1959), 303--306. Reprinted in: The collected works of Julia Robinson, Collected Works {\bf 6}, American Mathematical Society, Providence, RI (1996).


\bibitem[Rob62]{Rob62} Robinson, J., \emph{On the decision problem for algebraic rings}, The collected works of Julia Robinson, Collected Works {\bf 6}, American Mathematical Society, Providence, RI (1996). 

\bibitem[Rot95]{Rotman} J. J. Rotman, \textit{An introduction to the theory of groups}. Fourth edition. Graduate Texts in Mathematics, {\bf 148}. Springer-Verlag, New York, 1995. xvi+513 pp. ISBN: 0-387-94285-8. 

\bibitem[Sm23]{Sm23} Smith, Hanson, \emph{Radical Dynamical Monogenicity}, arXiv:2306.11815 (2023). 

\bibitem[St92]{St92} M. Stoll, \textit{Galois Group over $\Q$ of some iterated polynomials}.  Arch. Math. {\bf 59}, 239--244 (1992).

\bibitem[VV15]{VV15} Vidaux, X.; Videla, C. R. \textit{Definability of the natural numbers in totally real towers of nested square roots}, Proc. Amer. Math. Soc. {\bf 143} (2015), 4463--4477.

\bibitem[VV16]{VV16} X. Vidaux and C. R. Videla, \textit{A note on the Northcott property and undecidability}, Bull. London Math. Soc. {\bf 48}, 58--62 (2016). 

\bibitem[Wi16]{Wi16} M. Widmer, \emph{Property (N), Decidability, and Diophantine Approximation}, Oberwolfach Reports (2016), \url{https://www.mfo.de/occasion/1615}

\bibitem[Ya20]{Ya20} Yamamoto, Kota, \emph{On iterated extensions of number fields arising from quadratic polynomial maps.}
J. Number Theory 209 (2020), 289--311. 

\end{thebibliography}
\end{document}